\documentclass[11pt]{amsart}
\usepackage[usenames]{color}
\usepackage{fullpage}
\usepackage{amscd}
\usepackage{amssymb, latexsym}
\usepackage[all,2cell,ps]{xy}
\usepackage{mathdots}

\theoremstyle{plain}
\newtheorem{thm}{Theorem}[section]
\newtheorem{theorem}[thm]{Theorem}

\newtheorem{corollary}[thm]{Corollary}

\newtheorem{lemma}[thm]{Lemma}

\newtheorem{proposition}[thm]{Proposition}

\theoremstyle{definition}
\newtheorem{de}[thm]{Definition}

\newtheorem{example}[thm]{Example}

\newcommand{\Dis}{\rm Dis}
\newcommand{\Z}{\mathbb{Z}}
\newcommand{\N}{\mathbb{N}}

\newcommand{\Aff}{\mathrm{Aff}}

\newcommand{\aut}[1]{\mathrm{Aut}(#1)}
\newcommand{\dis}[1]{\mathrm{Dis}(#1)}
\newcommand{\lmlt}[1]{\mathrm{LMlt}(#1)}

\newcommand{\orb}[2]{\mathrm{Orb}_{#1}(#2)}

\numberwithin{equation}{section}

\begin{document}

\title{Free medial quandles}

\author{P\v remysl Jedli\v cka}
\author{Agata Pilitowska}
\author{Anna Zamojska-Dzienio}

\address{(P.J.) Department of Mathematics, Faculty of Engineering, Czech University of Life Sciences, Kam\'yck\'a 129, 16521 Praha 6, Czech Republic}
\address{(A.P., A.Z.) Faculty of Mathematics and Information Science, Warsaw University of Technology, Koszykowa 75, 00-662 Warsaw, Poland}

\email{(P.J.) jedlickap@tf.czu.cz}
\email{(A.P.) apili@mini.pw.edu.pl}
\email{(A.Z.) A.Zamojska-Dzienio@mini.pw.edu.pl}

\date{\today}

\begin{abstract}
This paper brings the construction of free medial quandles
as well as free $n$-symmetric medial
quandles and free $m$-reductive medial quandles.
\end{abstract}

\maketitle

\section{Introduction}\label{sec1}
A binary algebra $(Q,\cdot)$ is called a \emph{rack} if the
following conditions hold, for every $x,y,z\in Q$:
\begin{itemize}
\item $x(yz)=(xy)(xz)$ (we say $Q$ is \emph{left distributive}),
\item the equation $xu=y$ has a unique solution $u\in Q$ (we say $Q$ is a \emph{left quasigroup}).
\end{itemize}
An \emph{idempotent} rack is called a \emph{quandle} (we say $Q$ is idempotent if $xx=x$ for every $x\in Q$). A quandle $Q$ is \emph{medial} if, for every $x,y,u,v\in Q$, $$(xy)(uv)=(xu)(yv).$$

An important example of a medial quandle is an abelian group~$A$ with an
operation $\ast$ defined by
$a\ast b=(1-h)(a)+h(b)$,
where $h$ is an automorphism of~$A$.
This construction is called an {\em affine} quandle (or sometimes an {\em Alexander quandle})
and denoted by $\Aff(A,h)$.
In the literature \cite{Hou1,Hou2},
the group~$A$ is usually considered to be a
$\Z[t,t^{-1}]$-module, where~$t\cdot a=h(a)$, for each~$a\in A$.
We shall adopt this point of view here as well and we usually write $\Aff(A,r)$ instead, where $r$ is a ring element.

Note that in universal algebra terminology an algebra is said to be \emph{affine} if it is polynomially equivalent to a module. A subreduct of an affine algebra is called a \emph{quasi-affine algebra}, see e.g. \cite{KK}. Clearly, affine quandles are quasi-affine algebras.

Medial quandles lie in the intersection of the class of quandles and the class of \emph{modes} \cite{RS}. Recent development in quandle theory is motivated by knot theory
(see e.g. \cite{Car,EN}). The \emph{knot quandle} is a very powerful
knot invariant. Quandles have also applications in differential geometry \cite{Loo} and graph
theory \cite{DK}. Modes are generally idempotent and entropic algebras (algebras with a commutative clone of term operations). Mediality is {{an} other name for entropicity in the binary case.
For a more detailed history of medial quandles we refer to~\cite{JPSZ}.

Medial quandles do not form a variety of binary algebras, unless we
introduce an additional binary operation $\backslash$
and the identities
$$ x\backslash (x\ast y) \approx y \qquad\text{ and }\qquad x\ast (x\backslash y)\approx y$$
to define the left quasigroup property equationally.

The structure of free medial quandles remained open for a long time.
There were only results about more general classes, i.e.
general free modes were investigated by Stronkowski~\cite{Str}
and general free quandles by
Joyce~\cite{J82} and Stanovsk\'y~\cite{St}. Just recently in \cite{FGG} free racks and quandles were described but the characterization does not give any useful information about their structure.

There were also some special cases described, like \emph{involutory} medial quandles,
that means medial quandles satisfying  additionally $x*(x*y)\approx y$.
\begin{theorem}[{\cite[Theorem 10.5]{J82}}]\label{Th:Joyce}
 Let $n\in\N$.
 Denote by $F$ the subset of the quandle $\Aff(\Z^n,-1)$ consisting of those $n$-tuples where
 at most one coordinate is odd.
 Then $(F,\ast)$ is a free $(n+1)$-generated involutory medial quandle
 over $\{(0,\ldots,0), (1,0,\ldots,0), (0,1,0,\ldots,0), \ldots,(0,\ldots,0,1)\}$.
\end{theorem}

Here we generalize the result of Joyce but not directly.
We choose the path started in~\cite{JPSZ} instead
and we study certain permutation groups acting on quandles, called {\em displacement} groups.
It turns out that, in the case of free medial quandles, these groups are free $\Z[t,t^{-1}]$-modules
and we can construct the free medial quandles based on these modules.
Another important result is that the free medial quandles
embed into affine quandles. This shows that the variety of medial quandles
is generated by affine quandles.

Next we focus on two special classes: $n$-symmetric and
$m$-reductive medial quandles {which play a significant role within the class of finite medial quandles}.
A quandle $(Q,\ast)$ is $n$-\emph{symmetric}, if it satisfies
the identity
$$\underbrace{x\ast(x\ast\cdots \ast(x}_{n-times}\ast y)\cdots )\approx y.$$
We construct here free $n$-symmetric medial quandles and
we prove that free $n$-symmetric
quandles embed into products of affine quandles over modules over Dedekind domains.
This is useful especially when studying finite medial quandles
since each finite left quasigroup is $n$-symmetric, for some natural number $n$.

A quandle $(Q,\ast)$ is $m$-\emph{reductive}, if it satisfies the identity
$$(\cdots (x\ast \underbrace{y)\ast\cdots\ast y)\ast y}_{m-times}\approx y.$$
A quandle is called reductive if it is $m$-reductive, for some $m\in\N$.
Reductivity turns out to be a very important notion in the study of medial quandles as
each finite medial quandle embeds into a product of a reductive quandle and a quasigroup \cite{JPZ}.

The paper contents four Sections. In Section 2 we recall and present some facts about general medial quandles. Section 3 contains {{the} main results. Theorem \ref{th:main} gives a description of the free medial quandles and Theorem \ref{th:alter} a construction of affine quandles into which the free quandles embed. Section 4 is devoted to $n$-symmetric and $m$-reductive free medial quandles.
In both cases, the displacement group of the free algebra turns out to be a free $\Z[t]/(f)$-module,
for a suitable polynomial~$f$. The description of the free quandles in these varieties
is analogous as for general medial quandles.

Note that when studying left quasigroups, important tools are the mappings $L_e:x\mapsto e\ast x$,
called the left translations. We use also the right translations $R_e:x\mapsto x\ast e$.
The idempotency and the mediality imply that both $L_e$ and $R_e$ are endomorphisms.
The left quasigroup property means that $L_e$ is an automorphism.

\section{Preliminaries}

This section recalls some important notions from~\cite{JPSZ} where
the structure of medial quandles was described. Key ingredients are two permutation
groups acting on quandles.

\begin{de}
 Let~$Q$ be a quandle. The {\em left multiplication group} of~$Q$ is the group
 $$\lmlt{Q}=\left< L_x;\ x\in Q\right>.$$
 The {\em displacement group} is the group
 $$\dis{Q}=\left<L_xL_y^{-1};\ x,y\in Q\right>.$$
\end{de}

It was proved in~\cite[Proposition 2.1]{HSV} that the actions of both groups on~$Q$ have the same orbits.
We use, in the sequel, the word {\em orbit} plainly without explicitly
mentioning the acting groups. The orbit of~$Q$ containing~$x$ is denoted by ~$Qx$ and the stabilizer subgroup of $x$ is denoted by $\dis{Q}_x$. For two permutations $\alpha,\beta$, we write $\alpha^{\beta}=\beta\alpha\beta^{-1}$. The commutator is defined by $[\beta,\alpha]=\alpha^\beta\alpha^{-1}$.
The identity permutation is denoted by~$1$.

It is also useful to understand the structure of the displacement group.

\begin{lemma}[{\cite[Proposition 2.1]{HSV}}]\label{lem:dis_struct}
 $\dis{Q}=\{L_{x_1}^{\varepsilon_1}L_{x_2}^{\varepsilon_2}\cdots L_{x_k}^{\varepsilon_k};\
 x_i\in Q, \varepsilon_i=\pm1, \sum \varepsilon_i=0\}$.
\end{lemma}

From this lemma, we can clearly see that $\dis{Q}$ is a normal subgroup of~$\lmlt{Q}$.
Moreover, in our context, the group is commutative.

\begin{proposition}[{\cite[Proposition 2.4]{HSV}}]
 Let~$Q$ be a quandle. Then $Q$ is medial if and only if~$\dis{Q}$ is abelian.
\end{proposition}

Since $\dis{Q}$ is abelian, conjugations by elements from the same coset of~$\dis{Q}$ yield
the same results.

\begin{lemma}\label{lem:conjug}
 Let~$Q$ be a medial quandle. Let~$\alpha\in\dis{Q}$ and $x,y\in Q$. Then $\alpha^{L_x}=\alpha^{L_y}$.
\end{lemma}

\begin{proof}
 $\alpha^{L_x}=L_x\alpha L_x^{-1}=L_x\alpha L_x^{-1}L_yL_y^{-1}=L_xL_x^{-1}L_y\alpha L_y^{-1}=\alpha^{L_y}$
 due to the abelianess of~$\dis{Q}$.
\end{proof}

From now on, by writing $\alpha^L$, we mean $\alpha^{L_x}$, for an arbitrary $x\in Q$, since
the conjugation does not depend on the element~$x$.

It is easy to see that, for $\alpha\in\aut{Q}$ and $x\in Q$, $L_{\alpha(x)}=L_x^\alpha$.
In particular, for $\alpha=L_y$, we obtain ${L_{y\ast x}}=L_x^{L_y}$.
This implies that $\lmlt{Q}$ has only few generators. On the other hand, $\Dis{Q}$, in spite of being
a subgroup of~$\lmlt{Q}$, has usually more generators than $\lmlt{Q}$.

\begin{proposition}\label{prop:generators}
 Let~$Q$ be a medial quandle generated by~$X\subset Q$ and choose $z\in X$. Then
 \begin{itemize}
  \item the group $\lmlt{Q}$ is generated by $\{L_x;\ x\in X\}$;
  \item the group $\dis{Q}$ is generated by $\{ (L_xL_z^{-1})^{L^k};\ x\in X,\ k\in\Z\}$.
 \end{itemize}
\end{proposition}

\begin{proof}
 The generating set for $\lmlt{Q}$ is obtained by the induction using ${L_{x\ast y}}=L_xL_yL_x^{-1}$
 and $L_{x\backslash y}=L_x^{-1}L_yL_x$.

 Suppose now $\alpha\in\dis{Q}$. By Lemma~\ref{lem:dis_struct} and the previous observation
 we can suppose $\alpha=L_{x_1}^{\varepsilon_1}L_{x_2}^{\varepsilon_2}\cdots L_{x_n}^{\varepsilon_n}$,
 where $x_i\in X$ and $\varepsilon_i=\pm 1$ with $\sum \varepsilon_i=0$, for $1\leq i\leq n$.
 We prove the claim by an induction on~$n$. For $n=2$ the claim is true.

 Let the induction hypothesis holds for all words of length at most $n-2$.
 If $\varepsilon_1=\varepsilon_n$ then clearly $w=w_1w_2$ with $w_i\in\dis{Q}$ and we use the
 induction hypothesis. Let now $\varepsilon_1=1$ and $\varepsilon_n=-1$.
 Then $w=L_{x_1}w'L_{x_n}^{-1}$ and $w'$ is, by the induction hypothesis, a product of elements from
 $\{ (L_xL_z^{-1})^{L^k};\ x\in X,\ k\in\Z\}$. But
 $$w=L_{x_1}w'L_{x_1}^{-1}L_{x_1}L_{z}^{-1}L_z L_{x_n}^{-1}={(w')^L  (L_{x_1}L_{z}^{-1}) (L_{x_n}L_z^{-1})^{-1}}$$
 proving the claim. The argument is similar for $\varepsilon_1=-1$ and $\varepsilon_n=1$.
\end{proof}

This result cannot be much improved -- it is shown in Proposition~\ref{prop:free} that
the displacement group of a free medial quandle is not finitely generated.

The abelian group $\dis{Q}$ can be easily endowed with the structure of a $\Z[t,t^{-1}]$-module,
it suffices to pick an automorphism of~$\dis{Q}$. A natural choice is the inner automorphism
$\alpha\mapsto\alpha^L$. Hence, from now on, the group $\dis{Q}$ is treated, depending
on the situation, either
 as a permutation group acting on~$Q$ or
as an $R$-module, where $R$ is a suitable image of~$\Z[t,t^{-1}]$,
with the action of $t$ defined by $\alpha^t=\alpha^{L}$.
Note that, for $f\in \Z[t,t^{-1}]$, $\alpha^{f} =\alpha ^{f(L)}$.
\begin{example}
Let $f=(1-t)^2$. Then $\alpha^{f}=\alpha^{1-2t+t^2}=\alpha^{f(L)}=
\alpha^{1-2L+L^2}=\alpha(\alpha^{L})^{-2}\alpha^{L^2}$.
\end{example}

It was proved in \cite[Proposition 3.2]{JPSZ} that, for any $x\in Q$, the orbit $Qx$ is affine over $\dis{Q}/\dis{Q}_x$
and we can naturally identify the sets $Qx$ and $\dis{Q}/\dis{Q}_x$ by
defining the group operation on~$Qx$ as:
$$\alpha(x)+\beta(x)=\alpha\beta(x)\qquad \text{ and }\qquad-\alpha(x)=\alpha^{-1}(x).$$
The group so defined is denoted by $\orb{Q}{x}$ and called the {\em orbit group} for $Qx$.
Moreover, $\dis{Q}_x$ is a submodule of $\dis{Q}$: suppose $\alpha(x)=x$; then
${\alpha^t(x)}=L_x\alpha L_x^{-1}(x)=x$.
This means that $\dis{Q}/\dis{Q}_x$ is a $\Z[t,t^{-1}]$-module and we can
call $\orb{Q}{x}$ the {\em orbit module} for~$Qx$.

\section{Free medial quandle}

In this section we present the free medial quandles. Regarding the generating set,
we see that, in any quandle~$Q$, for all $x,y\in Q$, ${y\ast x}\in Qx$ as well as $y\backslash x\in Qx$.
Hence each orbit has to contain at least one generator.

\begin{lemma}\label{lem:orbits}
 Let $Q$ be a quandle generated by~$X\subset Q$. Then the set $X\cap Qx$ is nonempty, for each $x\in Q$.
\end{lemma}

The following proposition characterizes the free medial quandles. Formally, it is pronounced
as a sufficient condition only but we {can} 
see in Theorem~\ref{th:main} that such an
object exists, making the condition necessary too.

\begin{proposition}\label{prop:free}
 Let $F$ be a medial quandle generated by a set~$X\subset F$. Choose $z\in X$ arbitrarily. Then
 $F$ is free over~$X$ if the following conditions are satisfied:
 \begin{enumerate}
  \item each element of~$X$ lies in a different orbit;
  \item $\dis{F}$ is a free $\Z[t,t^{-1}]$-module with $\{L_xL_z^{-1};\ x\in X\smallsetminus\{z\}\}$
  as a free basis;
  \item the action of $\dis{F}$ on~$F$ is free.
 \end{enumerate}
\end{proposition}

\begin{proof}
 Observe first that, for any $y\in F$, there exists exactly one $x\in X$ and exactly one $\alpha\in\dis{F}$
 such that $y=\alpha(x)$. Indeed, the existence of $x$, comes from Lemma~\ref{lem:orbits},
 and its uniqueness from~(1). The uniqueness of $\alpha$ is due to~(3).

 Let $Q$ be a medial quandle and let $Y\subset Q$. Let $\psi$ be a mapping $X\to Y$.
 We prove that $\psi$ can be extended to a homomorphism $\Psi:F\to Q$.
 We define first a $\Z[t,t^{-1}]$-module homomorphism $\Phi:\dis{F} \to \dis{Q}$
 on the basis of~$\dis{F}$ by setting $\Phi(L_xL_z^{-1})=L_{\psi(x)}L_{\psi(z)}^{-1}$.
 Note that $\Phi(\alpha^L)=\Phi(\alpha^t)=\Phi(\alpha)^t=\Phi(\alpha)^L$.
 Now set
 $$\Psi(\alpha(x))=\Phi(\alpha)(\psi(x)),\qquad\text{ for all }\alpha\in\dis{F}\text{ and }x\in X.$$
 Mapping~$\Psi$ is well defined since every element of $F$ has a unique representation by~$\alpha$ and~$x$.
\begin{multline*}
  {\Psi(\alpha(x))\ast\Psi(\beta(y))=\Phi(\alpha)(\psi(x)) \ast \Phi(\beta)(\psi(y))}
  =L_{\Phi(\alpha)(\psi(x))}\Phi(\beta)(\psi(y))=L_{\psi(x)}^{\Phi(\alpha)} \Phi(\beta)(\psi(y))=\\
  \Phi(\alpha)L_{\psi(x)}\Phi(\alpha^{-1})\Phi(\beta)L_{\psi(x)}^{-1}L_{\psi(x)}L_{\psi(y)}^{-1}(\psi(y))
  =\Phi(\alpha)(\Phi(\alpha^{-1}\beta))^L L_{\psi(x)}L_{\psi(y)}^{-1} (\psi(y))=\\
  \Psi(\alpha(\alpha^{-1}\beta)^L L_xL_y^{-1}(y))=\Psi(\alpha L_x \alpha^{-1}\beta (y))
  =\Psi(L_{\alpha(x)}\beta(y))={\Psi(\alpha(x)\ast \beta(y))}
 \end{multline*}
 and $\Psi$ is a homomorphism that extends $\psi$.
\end{proof}

In the sequel, we use the following notation:
let~$X$ be a set. We choose $z\in X$ arbitrarily and
we denote by $X^{-}$ the set $X\smallsetminus\{z\}$.
We often do not specify the element~$z$ since we actually rarely need it explicitly.
Let now $R$ be a ring and consider the free $R$-module of rank $|X^-|$, i.e.
$M=\bigoplus_{x\in X^-} R$. We then choose a free basis of $M$, let us say $\{e_i;\ i\in X^-\}$, and by defining $e_z=0\in M$, we have defined $e_i$ as an element of~$M$, for each~$i\in X$.

\begin{theorem}\label{th:main}
 Let $X$ be a set and let $z\in X$. Denote by $X^-$ the set $X\smallsetminus\{z\}$.
 Let $M=\bigoplus_{x\in X^-}\Z[t,t^{-1}]$.
Let $\{e_i;\ i\in X^-\}$ be a free basis of~$M$. Moreover, let $e_z=0\in M$.
Let us denote by $F$ the set $M\times X$ equipped
 with the operation
 $$(a,i){\ast} (b,j)=((1-t)\cdot a+t\cdot b+e_i-e_j,j).$$
 Then $(F,{\ast})$ is a free
 medial quandle over $\{(0,i);\ i\in X\}$.
\end{theorem}

\begin{proof}
The idempotency is evident. The mediality is proved by {the observation that} 
$$((a,i)\ast(b,j))\ast((c,k)\ast(d,n))=((1-t)^2\cdot a+(t-t^2)\cdot(b+c)+t^2\cdot d+(1-t)\cdot e_i+t\cdot(e_j+e_k)-(1+t)\cdot e_n,n).$$
The {left-}quasigroup operation is given by the formula
$$
(a,i)\backslash (b,j)=((1-t^{-1})\cdot a+t^{-1}\cdot (b +e_j-e_i),j).
$$
Hence $F$ is a medial quandle.

We know now that $F$ is a medial quandle and we want to prove its freeness by Proposition~\ref{prop:free}.

We start with analyzing the structure of~$\dis{F}$.
\begin{multline*}
L_{(a,i)}L_{(b,j)}^{-1}((c,k))=(a,i){\ast}((b,j)\backslash(c,k))=(a,i){\ast} ((1-t^{-1})\cdot b+t^{-1}\cdot (c+e_k-e_j),k)=\\
((1-t)\cdot a +(t-1)\cdot b+c+e_k-e_j+e_i-e_k,k)=(c+(1-t)\cdot(a-b)+e_i-e_j,k).
\end{multline*}
In particular, $L_{(0,i)}L_{(0,z)}^{-1}((c,k))=(c+e_i,k)$. Now we prove by induction that
$$\big(L_{(0,i)}L_{(0,z)}^{-1}\big)^{L^n} ((c,j))=(c+t^n\cdot e_i,j), \qquad \text{for each }i\in X^-,j\in X\text{ and }n\in\Z.
$$
The case $n=0$ was already proved. Now suppose $n>0$.
\begin{multline*}
 \big(L_{(0,i)}L_{(0,z)}^{-1}\big)^{L^n} ((c,j))=L_{(0,z)}\big(L_{(0,i)}L_{(0,z)}^{-1}\big)^{L^{n-1}}L_{(0,z)}^{-1}((c,j))
 =\\ L_{(0,z)}\big(L_{(0,i)}L_{(0,z)}^{-1}\big)^{L^{n-1}}{(}(t^{-1}\cdot(c+e_j),j){)}
 =L_{(0,z)}{(}(t^{-1}\cdot(c+e_j)+t^{n-1}\cdot e_i,j){)}=(c+t^n\cdot e_i,j).
\end{multline*}
The case $n<0$ is analogous. Moreover, from this we see that $(L_{(0,i)}L_{(0,z)}^{-1})^f\ {(}(c,j){)}=(c+f\cdot e_i,j)$,
for any $c\in M$, $f\in \Z[t,t^{-1}]$ and $i,j\in X$.

Let $\big((f_i)_{i\in X^-},j\big)$ be in $F$.
We {now} prove that this element lies in the subquandle generated by $\{(0,i);\ i\in X\}$.
But it is not difficult to see that
$$\big((f_i)_{i\in X^-},j\big)=\prod_{i\in X^-} \big(L_{(0,i)}L_{(0,z)}^{-1}\big)^{f_i} \ {(}(0,j){)}.$$
The product is finite since only finitely many $f_i$ are non-zero. Hence $\left< \{(0,i);\ i\in X\}\right>=F$.
Moreover, we see that different generators lie in different orbits.

Since {the set} $\big\{\big(L_{(0,i)}L_{(0,z)}^{-1}\big)^{L^n}{; i\in X, n\in \mathbb{Z}}\big\}$ generates $\dis{F}$, due to Proposition~\ref{prop:generators},
we see that $\dis{F}$ acts freely on every orbit of~$F$. That means also that $\dis{F}$ is isomorphic to $M$
and $\{L_{(0,i)}L_{(0,z)};\ i\in X^-\}$ is clearly its free basis. According to Proposition~\ref{prop:free}, $F$ is free over~$\{(0,i);\ i\in X\}$.
\end{proof}

In~\cite{JPSZ}, the structure of medial quandles was represented using a heterogeneous structure
called the indecomposable affine mesh. We do not recall the definition here as it is not needed,
we just remark that the free medial quandle now constructed is the sum of the affine mesh
$$\Big( (\bigoplus_{x\in X^-} \Z[t,t^{-1}] )_{i\in X};\
(1-t)_{i,j\in X};\ (e_i-e_j)_{i,j\in X} \Big).$$

Recall that subquandles of affine quandles are quasi-affine.
Every (both sided) cancellative medial quandle is quasi-affine -- to see
this we can either use a result by Kearnes~\cite{KK}
for idempotent cancellative algebras having a central
binary operation or a
result by Romanowska and Smith {(see e.g.~\cite{RS})}
for cancellative modes.
Nevertheless,
a direct proof is simple.

\begin{proposition}
 Let $Q$ be a cancellative medial quandle. Then $Q$ embeds into any of its orbits.
\end{proposition}

\begin{proof}
 $R_x$ is an endomorphism of~$Q$, for each $x\in Q$.
 The right cancellativity ensures that $R_x$ is injective.
 Hence, for each $x\in Q$, $R_x$ embeds $Q$ into $Qx$.
\end{proof}

The free medial quandle, we have constructed, is cancellative and therefore it can be
represented as a~subquandle of an affine quandle.

{\begin{theorem}\label{th:alter}
 Let $X$ be a set. The free medial quandle over~$X$ is isomorphic to a subquandle of the affine quandle
 $M=\Aff(\bigoplus_{x\in X^-}\Z[t,t^{-1}],t)$.
 \end{theorem}
 }

\begin{proof}
{Let $\Lambda:\bigoplus_{x\in X^-}\Z[t,t^{-1}]\to\bigoplus_{x\in X^-}\Z$ be the group
 homomorphism induced by the evaluation homomorphism $\Z[t,t^{-1}]\to\Z${;} $t\mapsto 1$.
 Denote by $Q=\{a\in M;\ \Lambda(a)=e_i,\ \text{for some }i\in X\}$. Then, $Q$ is a subquandle of $M$, since
 $$\Lambda(L_a(b))=\Lambda((1-t)\cdot a+t\cdot b)=0\cdot\Lambda(a)+1\cdot\Lambda(b)=\Lambda(b)$$
 and analogously $\Lambda(L_a^{-1}(b))=\Lambda(b)$.
 We shall prove that $Q$ is a free quandle over $\{e_x;\ x\in X\}$.}

 Take the quandle~$F$ from Theorem~\ref{th:main}. Note that the orbit {$F(0,z)$ is isomorphic to $M$} through the bijection $(a,z)\mapsto a$.
 Now consider the embedding $R_{(0,z)}:F\to F(0,z)$. Clearly $R_{(0,z)}((0,i))=(e_i,0)$.
 Therefore the subquandle of $M$ generated by~$Y=\{e_x;\ x\in X\}$ is free.

 The only thing left to prove is { to show that} $Q=\left<Y\right>$. {Clearly $Y\subset Q$ and $Q$ is a subquandle of $M$, hence $Q\supseteq\left<Y\right>$}. 

 On the other hand, for $Q\subseteq\left<Y\right>$,
 we notice that $Q=\{a\in M;\ a\equiv e_i\pmod{(1-t)},\ \text{for some }i\in X\}$.
 Moreover, we have $(L_{e_x}L_{e_z}^{-1})^{L^n}:Q\rightarrow Q;\;u\mapsto u+(1-t) t^n\cdot e_x$, with
 an analogous proof as in Theorem~\ref{th:main}. Then $(L_{e_x}L_{e_z}^{-1})^{f(L)}(u)=u+(1-t)f\cdot e_x$, for each $f\in\Z[t,t^{-1}]$. Now, for each element $a\in Q$, we have $a=e_i+(1-t) g$, for some
 $i\in X$ and $g=(g_x)_{x\in X^-}\in M$. Hence, we have
 $$a=\prod_{x\in X^-}\big(L_{e_x}L_{e_z}^{-1}\big)^{g_x(L)} (e_i).$$

Therefore $a\in  \left<Y\right>$.
\end{proof}

{\begin{example}
We describe now the free medial quandle on three generators.
Let $X=\{0,1,2\}$, let $e_1=(1,0)$, $e_2=(0,1)$ and $e_0=(0,0)$.
Now $M=\Aff(\Z[t,t^{-1}]\times \Z[t,t^{-1}],t)$ and
$\Lambda:\Z[t,t^{-1}]\times \Z[t,t^{-1}]\to\Z^{2}$ is the group
homomorphism induced by the evaluation homomorphism $\Z[t,t^{-1}]\to\Z${;} $t\mapsto 1$.
Denote by $Q=\{a\in M;\; \Lambda(a)\in\{(0,0),(1,0),(0,1)\}\}$.
By Theorem \ref{th:alter}, $Q$ is a subquandle of $M$ and it is the free quandle generated by the set $\{(0,0),(1,0),(0,1)\}$.
For example, the element $a=(1-t,1+t-t^2)$ lies in~$Q$ since $\Lambda(a)=(0,1)$. Now~$a$
can be represented as $(0,1)+{(1-t)(1,t)}=(0,1)+(1-t)\cdot(1,0)+(1-t)t\cdot(0,1)=(L_{(1,0)}L_{(0,0)}^{-1})(L_{(0,1)}L_{(0,0)}^{-1})^L(0,1)=(L_{e_1}L_{e_0}^{-1})(L_{e_2}L_{e_0}^{-1})^L(e_2).$
\end{example}}

\section{Free quandles in subvarieties}

In this section we study free $n$-symmetric and free $m$-reductive medial quandles.
Both types of varieties have a similar property: they can be characterized by
an identity on $\dis{Q}$.

\begin{de}
 Let $I\subset \Z[t,t^{-1}]$. We say that a medial quandle~$Q$ is an $I$-quandle,
 if ${\alpha^f}=1$, for each $\alpha\in\dis{Q}$ and $f\in I$.
\end{de}

If~$I$ is an ideal of~$\Z[t,t^{-1}]$
and a $\Z[t,t^{-1}]$-module $M$ satisfies the identity $f\cdot a=0$,
for each~$a\in M$ and~$f\in I$,
then $M$ can be viewed as a module over $\Z[t,t^{-1}]/I$.

In our context, the set~$I$ shall usually be a principal ideal,
that means~$I=(f)$, for some $f\in\Z[t,t^{-1}]$.
We then write that $Q$ is an $f$-quandle, rather than $\{f\}$-quandle
or $(f)$-quandle.
If, moreover, $f=\sum_{r=0}^n c_rt^r$ and the coefficient $c_0$ is invertible
then $\Z[t,t^{-1}]/f\cong\Z[t]/f$ since
$t^{-1}\equiv -\big( \sum_{r=1}^n c_rt^{r-1} \big)\cdot c_0^{-1}\pmod{f}$.
We use these remarks since working with the ring
$\Z[t]/f$ is often easier than working with the ring $\Z[t,t^{-1}]$.

We prepared the framework of $I$-quandles to work with symmetric and reductive medial quandles at once.
First note that if $Q$ is an $I$-quandle then clearly $\orb{Q}{x}^I=\{\alpha^f(x)\mid \alpha\in \dis{Q},\; f\in I\}=\{x\}$. On the other hand, let $\alpha^f(x)=x$ for arbitrary
$\alpha\in\dis{Q}$, $f\in I$ and each $x\in Q$. Hence the action of $\alpha^f$ on~$Q$ is trivial and this means $\alpha^f=1$ since
 $\dis{Q}$ is faithful. This immediately gives the following lemma.

\begin{lemma}\label{lem:4.1}
 Let $Q$ be a medial quandle and let  $I\subset \Z[t,t^{-1}]$. Then $Q$ is an $I$-quandle
 if and only if $\orb{Q}{x}^I=\{x\}$, for each $x\in Q$.
\end{lemma}

Recall that a quandle $Q$ is $n$-symmetric, if $L_x^n=1$, for each $x\in Q$.

\begin{proposition}
 A medial quandle $Q$ is $n$-symmetric if and only if
 $Q$ is a $(\sum_{r=0}^{n-1} t^r)$-quandle.
\end{proposition}

\begin{proof}
 According to \cite[Proposition 7.2]{JPSZ},
 a medial quandle~$Q$ is $n$-symmetric if and only
 if, for each $x\in Q$
 $$\sum_{r=0}^{n-1}(1-\varphi)^{r}{\alpha(x)}=x,$$
where $\varphi:Qx\to Qx$ is defined
 by $\alpha(x)\mapsto [\alpha,L](x)$.

 Then $(1-\varphi)(\alpha(x))=\alpha(x)-[\alpha,L](x)=\alpha[L,\alpha](x)=\alpha^L{(x)}{=\alpha^t(x)}$
 and therefore it means that $\alpha^{\sum_{r=0}^{n-1}t^r}(x)=x$,
 for each $x\in Q$. Hence, according to Lemma~\ref{lem:4.1},
 $Q$ is $n$-symmetric if and only if it is a $(\sum_{r=0}^{n-1} t^r)$-quandle.
\end{proof}

Recall that a quandle $Q$ {is} $m$-reductive, if $R_x^m(y)=x$, for all $x,y\in Q$.

\begin{proposition}
 A medial quandle $Q$ is $m$-reductive
  if and only if it is a $(1-t)^{m-1}$-quandle.
\end{proposition}

\begin{proof}
 According to \cite[Proposition 6.2]{JPSZ},
 a medial quandle~$Q$ is $m$-reductive if and only
 if
 $\varphi^{m-1}{\alpha(x)}=x,$
 for each $x\in Q$, where $\varphi:Qx\to Qx$ is defined
 by $\alpha(x)\mapsto [\alpha,L](x)$.

 Then {$\varphi(\alpha(x))=[\alpha,L](x)=\alpha(\alpha^{-1})^L(x)=\alpha^{(1-t)}(x)$}
 and therefore it means that {$\alpha^{(1-t)^{m-1}}(x)=x$},
 for each $x$. Hence, according to Lemma~\ref{lem:4.1},
 $Q$ is {$m$-reductive if and only if it is} a $(1-t)^{m-1}$-quandle.
\end{proof}

Both mentioned polynomials, i.e. $\sum_{r=0}^{n-1} t^r$ and $(1-t)^{m-1}$ have the property
that the leading as well as the absolute coefficients are invertible.

In this case, not only $\dis{Q}$ can be treated as a $\Z[t]/f$-module but it has less generators even as a group.

\begin{proposition}
 Let $f=\sum_{r=0}^s c_rt^r$ be a polynomial with $c_0$ and $c_k$ invertible
 and let~$Q$ be an $f$-quandle generated by $X\subset Q$. Let $z\in X$ {be an arbitrary element}. Then $\dis{Q}$ is generated by
 $\{(L_xL_z^{-1})^{L^{r}}; \ x\in X\smallsetminus\{z\} \text{ and }0\leq r<s\}$.
\end{proposition}

\begin{proof}
 According to Proposition~\ref{prop:generators}, the group $\dis{Q}$ is
 generated by $(L_xL_z^{-1})^{L^{k}}$, for $x\in X\smallsetminus \{z\}$ and $k\in\Z$.
 But now
 \begin{align*}
  (L_xL_z^{-1})^{L^{-1}}&=(L_xL_z^{-1})^{t^{-1}}=(L_xL_z^{-1})^{-(\sum_{r=0}^{s-1} c_{r+1}t^r)c_0^{-1}}=\prod_{r=0}^{s-1}\nolimits \big((L_xL_z^{-1})^{L^r})^{-c_{r+1}c_0^{-1}},\\
  (L_xL_z^{-1})^{L^s}&=(L_xL_z^{-1})^{t^s}=(L_xL_z^{-1})^{-(\sum_{r=0}^{s-1} c_rt^r)c_s^{-1}}=\prod_{r=0}^{s-1}\nolimits \big((L_xL_z^{-1})^{L^r})^{-c_rc_s^{-1}}.
 \end{align*}
 Similarly for all $(L_xL_z^{-1})^{L^r}$,
 where $r<-1$ or~$r>s$.
\end{proof}

The structure of free medial $f$-quandles can be described exactly in the same way as
the structure of general free medial quandles.

\begin{proposition}\label{prop:free2}
 Let $f\in\Z[t]$ be a polynomial with the leading and the absolute coefficients invertible.
 Let $F$ be an $f$-quandle generated by a set~$X\subset F$. Choose $z\in X$ arbitrarily. Then
 $F$ is a free $f$-quandle over~$X$ if the following conditions are satisfied:
 \begin{enumerate}
  \item each element of~$X$ lies in a different orbit;
  \item $\dis{F}$ is a free $\Z[t]/f$-module with $\{L_xL_z^{-1};\ x\in X\smallsetminus\{z\}\}$
  as a free basis;
  \item the action of $\dis{F}$ on~$F$ is free.
 \end{enumerate}
\end{proposition}

\begin{proof}
 The proof is nearly the same as the proof of Proposition~\ref{prop:free}.
 The only difference is that displacements groups appearing here are $\Z[t]/f$-modules.
\end{proof}

\begin{theorem}\label{th:main2}
 Let $f\in\Z[t]$ be a polynomial with the leading and the absolute coefficients invertible.
 Let $X$ be a set.
 Let $M=\bigoplus_{{s}\in X^-}\Z[t]/f$. Let us denote by $F$ the set $M\times X$ equipped
 with the operation
 $$(a,i){\ast} (b,j)=((1-t)\cdot a+t\cdot b+e_i-e_j,j).$$
 Then $(F,{\ast})$ is a free
 $f$-quandle over $\{(0,i);\ i\in X\}$.
\end{theorem}

\begin{proof}
 The proof is nearly the same as of Theorem~\ref{th:main}, with the usage of Proposition~\ref{prop:free2}.
 The only thing {to show} is that ${\alpha^f}=1$, for all $\alpha\in\dis{F}$.
 But this follows from $\dis{F}\cong M$.
\end{proof}

Non-trivial reductive medial quandles are never right-cancellative since the multiplication
by $(1-t)$ is not injective.
On the other hand, free $n$-symmetric quandles
are cancellative and we can embed them in their orbits.
Moreover, the polynomial $\sum t^r$ is a product of cyclotomic polynomials.

\begin{theorem}\label{th:alter2}
 Let $X$ be a set.
 Let $n\in\N$ and let $\sum_{r=0}^{n-1} t^r=\prod_{j\in\mathcal{J}} f_j$,
 where $f_j$ are irreducible in~$\Z[t]$ and $M_j=\mathrm{Aff}(\bigoplus_{x\in X^-} \Z[t]/f_j,t)$, for each $j\in\mathcal{J}$ .
 Each free $n$-symmetric medial quandle is isomorphic to the subquandle of $\prod_{j\in\mathcal{J}} M_j$
 generated by $\{(e_i)_{j\in\mathcal{J}};\ i\in X\}$.
 \end{theorem}

\begin{proof}

We show that the subquandle $Q=\{(a_{j})_{j\in \mathcal{J}};\ \exists i\in X\ \forall j\in \mathcal{J}\ a_{j}\equiv e_i\pmod{(1-t)}\}$
 of $\prod_{j\in\mathcal{J}} M_j$ is a free $n$-symmetric medial quandle
 over $\{(e_i,e_i,\ldots,e_i);\ i\in X\}$.

 Denote by $f=\sum_{r=0}^{n-1} t^r$.
 It is well known that $t^n-1={(t-1) f=(t-1) \prod_{j\in\mathcal{J}} f_j}$
 and all the polynomials $f_j$ are pairwise different.
 Therefore, according to the Chinese remainder theorem, $\Z[t]/f\cong \bigoplus_{j\in\mathcal{J}} \Z[t]/f_j$.
 The rest is the same as in Theorem \ref{th:alter}.
\end{proof}

\begin{example}
 Consider $n=2$, i.e. involutory medial quandles. Since
 $\Z[t]/(t+1)\cong \Z$ and $t\equiv -1\pmod {(t+1)}$,
 according to Theorem~\ref{th:alter2},
 the free $|X|$-generated involutory medial quandle is isomorphic to the subquandle
 of $\Aff(\bigoplus_{{x}\in X^-} \Z,-1)$ that consists of those $|X^-|$-tuples
 congruent to some $e_i$ modulo~$2$. This confirms the result of Joyce (Theorem~\ref{Th:Joyce}).
\end{example}

Theorem~\ref{th:alter2} can be reformulated as follows: let $\zeta_k$ be the primitive $k$-th root
of unity in~$\mathbb{C}$. It is well known that $\Z[t]/(1+t+\cdots+t^{n-1})\cong \prod_{k|n,k>1} \Z[\zeta_k]$.
Hence the free $|X|$-generated $n$-symmetric quandle is the subquandle
of $\prod_{k|n,k>1} \Aff(\bigoplus_{{x}\in X^-} \Z[\zeta_k],\zeta_k)$ generated by $\{(e_i,e_i,\ldots,e_i);\ i\in X\}$.

\begin{example}
 Consider $|X|=2$. Then the free $2$-generated $n$-symmetric medial quandle {$F$} is
 the subquandle of $\prod_{k|n,k>1} \Aff(\Z[\zeta_k],\zeta_k)$ generated by $(0,\ldots,0)$
 and $(1,\ldots,1)$.
 The subquandle $F$ consists of those tuples $(a_k)_{k|n, k>1}$,
that either $a_k\equiv 0\pmod {(1-\zeta_k)}$,
 for all $k|n, k>1$,  or $a_k\equiv 1\pmod {(1-\zeta_k)}$,
 for all $k|n, k>1$.
\end{example}

Every finite quandle is $n$-symmetric, for some~$n$. For studying
finite medial quandles, it is nice to hear that we do not need always consider
$\Z[t,t^{-1}]$-modules but we can sometimes restrain or focus on nicer rings, or
even domains.

\begin{corollary}
 Let $n\in\N$. The variety of $n$-symmetric medial quandles is generated by quandles that
 are polynomially equivalent to modules over Dedekind domains.
\end{corollary}

\begin{proof}
 It is well known \cite{ZS75} that $\Z[\zeta_k]$ is a Dedekind domain, for each~$k$.
 Free $n$-symmetric quandles embed into products of $\Aff(\bigoplus{_{x\in X^-}} \Z[\zeta_k],\zeta_k)$.
 Each of these affine quandles embeds into any of its orbits, i.e. into a module over $\Z[\zeta_k]$.
\end{proof}

Note that applying our idea of $I$-quandles one obtains the description of free $n$-symmetric $m$-reductive medial quandles if
we consider $I=\{\sum_{r=0}^{n-1} t^r,(1-t)^{m-1}\}$. In particular, the free \mbox{2-reductive} $n$-symmetric
medial quandle over X is isomorphic to $\bigoplus_{{x}\in X^-}\Z_n\times X$
with the operation $(a,i)*(b,j)=(b+e_i-e_j,j)$
\cite[Proposition 2.4]{RR89}.

\end{document}